\documentclass[11pt]{article}

\usepackage[a4paper,margin=1in]{geometry}
\usepackage{amsmath,amssymb,amsfonts,amsthm}
\usepackage{bm}
\usepackage{mathtools}
\usepackage{graphicx}
\usepackage{enumitem}
\usepackage{hyperref}
\usepackage{color}
\usepackage{authblk}

\newcommand{\Lag}{\mathcal L}
\newcommand{\DLag}{\mathcal L^a}

\newtheorem{proposition}{Proposition}
\newtheorem{theorem}{Theorem}

\title{
One-Shot Optimization with Additional Inequality Constraints
}

\author[1]{Lea Fischer}
\author[2]{Nicolas R. Gauger}
\author[3]{Lisa Kusch}

\affil[1,2]{RPTU University Kaiserslautern-Landau, Kaiserslautern, Germany, nicolas.gauger(at)rptu.de}
\affil[3]{Eindhoven University of Technology, Eindhoven, The Netherlands, l.kusch(at)tue.nl}

\date{}

\begin{document}

\maketitle

\begin{abstract}

The one-shot approach is a powerful simultaneous optimization framework
for design tasks governed by computationally expensive steady-state
systems.
While previous formulations mainly focused on additional equality
constraints, this work extends the one-shot framework to optimization
problems with inequality constraints using slack variables embedded into
a doubly augmented Lagrangian formulation.

After elimination of the slack variables, the resulting formulation
contains nonsmooth active-set dependent terms while preserving the
characteristic coupled one-shot matrix structure.
The resulting generalized gradient system admits a representation of the
form
$
\nabla \DLag = -Ms,
$
analogous to the equality-constrained one-shot framework.

To analyze the nonsmooth active-set transitions, generalized second-order
subdifferentials in the sense of Clarke and Rockafellar--Wets are
employed.
We derive explicit positivity conditions ensuring positive definiteness of all admissible generalized Hessians and hence strict local optimality. 
Furthermore, we give an outlook on the construction of preconditioners based on generalized Hessian approximations.
\medskip


\medskip

\noindent
49M05, 65K05, 90C30

\end{abstract}

\section{Introduction}
\label{intro}

Simultaneous analysis and design (SAND) frameworks drastically reduce
the computational overhead of PDE-constrained optimization by advancing
state variables, adjoint variables, and design variables concurrently
within a single coupled iteration. One-shot methods constitute one of the most successful realizations of
this idea.
Instead of solving state and adjoint systems accurately within each
optimization step, one-shot methods update all variables simultaneously. The framework has proven particularly successful for large-scale
simulation-based optimization problems such as aerodynamic shape
optimization and PDE-constrained optimization problems
\cite{hazra2005,gauger2008,bosse2014}.

Extensions of one-shot methods to optimization problems involving
additional equality constraints were investigated in
\cite{walther2018, walther2018equality}. The method is based on the augmented Lagrangian $\DLag$, which is augmented by state and adjoint residuals, as well as the equality constraint. 
A key result therein is the representation
\begin{equation}
\nabla \DLag = -Ms,
\label{eq:introM}
\end{equation}
where $M$ has a characteristic coupled block structure and $s$ is the update step of the corresponding algorithm.

In many engineering applications, however, optimization problems involve additional inequality constraints of the form
\begin{equation}
k(y,u)\le 0.
\label{eq:ineq}
\end{equation}
Typical examples include stress constraints, positivity requirements, safety margins, and constraints arising from operational restrictions. 

Furthermore, engineering applications often require solving a multi-objective optimization problem, where the so-called constraint methods for scalarization of the multi-objective optimization problem can be employed. These methods find approximations to a set of Pareto optimal solutions that represent different trade-offs between the competing objective functions. The one-shot method has been used in \cite{kusch2018} to solve multi-objective optimization problems using the equality constraint method. However, for non-convex solutions sets, the equality constraint method cannot guarantee to find Pareto optimal solutions. Instead, we need to use methods based on inequality constraints, e.g., the epsilon-constraint method \cite{marglin1967}. 

The incorporation of inequality constraints into one-shot methods is nontrivial because active-set transitions induce nonsmoothness. The objective of this work is therefore to extend the doubly augmented
one-shot framework to optimization problems with additional inequality
constraints while preserving the characteristic coupled one-shot
structure. The key idea consists in introducing slack variables and subsequently
eliminating them analytically.
The resulting formulation naturally introduces generalized derivatives
and semismooth active-set structures.

\section{Optimization Problem}

We consider optimization problems of the discretized form
\begin{equation}
\begin{aligned}
\min_{y,u}\quad & f(y,u)
\\
\text{s.t.}\quad &
c(y,u)=0,
\\
&
k(y,u)\le 0.
\end{aligned}
\label{eq:problem}
\end{equation}

Here, $y\in \mathcal{Y} \subset \mathbb{R}^m$ is the vector of state variables and $u \in \mathcal{U} \subset \mathbb{R}^n$ is the vector of design variables. The discretized state equation is given by the equality constraints $c(y,u)=0$ with $c:\mathcal{Y}\times\mathcal{U} \rightarrow \mathbb{R}^m.$ The inequality constraints are given by $k: \mathcal{Y}\times\mathcal{U} \rightarrow \mathbb{R}^r$, where we assume that $r \leq m.$ 

We assume that the state equation admits a fixed-point formulation
\begin{equation}
y=G(y,u),
\label{eq:fixedpoint}
\end{equation}
with contractive state iteration satisfying
\[
\|G_y(y,u)\|\le \rho <1.
\]

In the following, we assume sufficient differentiability of all functions that are not considered in a separate analysis for non-smoothness. 

\section{Doubly Augmented Lagrangian Formulation}

The subsequent method and analysis are based on a double augmented Lagrangian formulation augments the standard Lagrangian of problem \eqref{eq:problem} by additional terms for the state and adjoint residuals, as well as the inequality constraints. The inequality constraints are introduced with the help of slack variables. 

\subsection{Slack Variable Formulation}

To handle the inequality constraints, slack variables
\[
l\ge 0
\]
are introduced such that
\begin{equation}
k(y,u)+l=0.
\end{equation}

The Lagrangian $\Lag(y,u,\lambda,\mu,l)$ of problem \ref{eq:problem} is given by
\begin{align}
    \Lag(y,u,\lambda,\mu,l) = f(y,u)
+ \lambda^\top(G(y,u)-y)
+ \mu^\top(k(y,u)+l)
\end{align}
with multipliers $\lambda \in \mathbb{R}^m$ and $\mu \in \mathbb{R}^r.$ 

We only consider the variables $ (y,u,\lambda,\mu)$. The first-order optimality conditions for this set of variables are given by 
\begin{subequations}\label{eq:fo_l}
\begin{align}
\nabla_y \Lag &= 0, \label{eq:fo_l_y}\\
\nabla_u \Lag &= 0, \label{eq:fo_l_u}\\
\nabla_\lambda \Lag &= G(y,u)-y = 0, \label{eq:fo_l_lambda}\\
\nabla_\mu \Lag &= k(y,u)+l = 0. \label{eq:fo_l_mu}
\end{align}
\end{subequations}

Let us define the shifted Lagrangian as
\begin{align*}
    \mathcal{N}(y,u,\lambda, \mu) = f(y,u)+\lambda ^T G(y,u)+\mu^Tk(y,u).
\end{align*}
Then the necessary conditions can be expressed as
\begin{subequations}\label{eq:fo_shifted}
\begin{align}
\nabla_y \mathcal{N}(y,u,\lambda,\mu) - \lambda &= 0, \label{eq:fo_shifted_y}\\
\nabla_u \mathcal{N}(y,u,\lambda,\mu) &= 0, \label{eq:fo_shifted_u}\\
G(y,u) - y &= 0, \label{eq:fo_shifted_state}\\
k(y,u) + l &= 0. \label{eq:fo_shifted_ineq}
\end{align}
\end{subequations}

The associated doubly augmented Lagrangian is defined by
\begin{align}
\label{eq:dal}
\DLag(y,u,\lambda,\mu,l)
&=
\Lag(y,u,\lambda,\mu,l)\\
&+
\frac{\alpha}{2}\|G(y,u)-y\|^2 +
\frac{\alpha}{2}\|k(y,u)+l\|^2 +
\frac{\beta}{2}
\|
\nabla_y\Lag(y,u,\lambda,\mu)
\|^2. \nonumber
\end{align}

For the augmented part, $\alpha >0$ is the associated multiplier of the state residual \eqref{eq:fo_l_lambda} and the slack-based constraint \eqref{eq:fo_l_mu} combined, and $\beta >0$ is the associated multiplier of the adjoint residual \eqref{eq:fo_l_y}. 

\subsection{Slack Elimination}

The first-order optimality condition of the doubly augmented Lagrangian \eqref{eq:dal} with respect to $l$ gives the KKT conditions

\begin{equation}
l \ge 0, ~\mu + \alpha (k+l) \ge 0, ~l_i\bigl(\mu_i+\alpha(k_i+l_i)\bigr)=0,
\quad i=1,\ldots,r.
\label{eq:slack_complementarity}
\end{equation}

From these we obtain
\begin{equation}
l = \max\left( 0, -\frac{\mu}{\alpha}-k \right).
\label{eq:slack}
\end{equation}

\begin{proposition}
\label{prop1}
Substituting \eqref{eq:slack} into the constraint residual yields
\[
k+l
=
\frac1\alpha
\left(
\max(0,\alpha k+\mu)-\mu
\right).
\]

\end{proposition}

\begin{proof}

From \eqref{eq:slack}
we obtain
\begin{align*}
    k+l=k+\max\left(0,-\frac{\mu}{\alpha}-k\right).
\end{align*}

Multiplying by $\alpha>0$ gives
\begin{align*}
\alpha(k+l) = \alpha k + \max(0,-\mu-\alpha k).
\end{align*}

Using the identity
$x+\max(0,-x)=\max(0,x)$,
with $x=\alpha k+\mu$, and dividing by $\alpha$ yields the result.

\end{proof}

We can use this elimination to reformulate the augmented Lagrangian.

\begin{proposition}

The inequality contribution to the augmented Lagrangian reduces to
\begin{equation}
\frac1{2\alpha}
\|
\max(0,\alpha k+\mu)
\|^2
-
\frac1{2\alpha}\|\mu\|^2.
\label{eq:maxterm}
\end{equation}

\end{proposition}

\begin{proof}

Let
\[
z:=\max(0,\alpha k+\mu).
\]

By Proposition \ref{prop1},
\[ k+l = \frac{1}{\alpha}(z-\mu). \]

Substituting this into the inequality contribution gives
\begin{align*}
\mu^{\top}(k+l) + \frac{\alpha}{2}\|k+l\|^2
&= \frac{1}{\alpha}\mu^{\top}(z-\mu) + \frac{1}{2\alpha}\|z-\mu\|^2.
\end{align*}

Expanding the square gives
\begin{align*}
\mu^{\top}(k+l) + \frac{\alpha}{2}\|k+l\|^2
&=  \frac{1}{\alpha}\mu^{\top}z -\frac{1}{\alpha}\|\mu\|^2
\\
&\quad +\frac{1}{2\alpha} \left(
\|z\|^2 -2\mu^{\top}z +\|\mu\|^2 \right).\\
& = \frac{1}{2\alpha}\|z\|^2 -\frac{1}{2\alpha}\|\mu\|^2.
\end{align*}

Substituting $z$ as defined above completes the proof.

\end{proof}

This contribution coincides with the classical
augmented-Lagrangian term for inequality
constraints \cite{rockafellar1973multiplier}.

\section{One-Shot Gradient Representation}

We will now show the specific structure of the gradient of the new doubly augmented Lagrangian.

Referring back to the first-order optimality conditions \eqref{eq:fo_shifted} and using Proposition \ref{prop1}, we can define an iterative procedure with an update step.

The one-shot step is then defined by
\begin{align}
    \label{eq:os_step}
    s=
\begin{pmatrix}
\Delta y 
\\
\Delta u
\\
\Delta \lambda
\\
\Delta \mu
\end{pmatrix}
=
\begin{pmatrix}
G(y,u)-y 
\\
\nabla_u \mathcal{N}(y,u,\lambda, \mu)
\\
\nabla_y \mathcal{N}(y,u,\lambda, \mu)-\lambda
\\
\frac{1}{\alpha}(\max(0,\alpha k+\mu)-\mu)
\end{pmatrix}.
\end{align}
To show the specific structure of the doubly augmented Lagrangian, we need to generalize the Jacobian due to the maximum function. 

\subsection{Generalized Active-Set Derivatives}
\label{sec:generalized}

Define
\[
z:=\alpha k+\mu.
\]

The max mapping
\[
\phi(z)=\max(0,z)
\]
is locally Lipschitz continuous and therefore directionally differentiable and semismooth in the sense of Clarke \cite{clarke1983}.

The Clarke generalized Jacobian $D(z)\in \partial \phi(z)$ is then given by $D(z)=\text{diag}(d_i)$ with 
\[
d_i \in
\begin{cases}
\{0\}, & z_i<0,
\\
[0,1], & z_i=0,
\\
\{1\}, & z_i>0.
\end{cases}
\]

For every $z$ and every selection $D(z)\in\partial\phi(z)$,
the componentwise identity 
\begin{align}
\label{eq:iden}
    \phi(z)=D(z)\,z
\end{align}
holds. This can be directly observed from the properties of the maximum function, so that the mapping $\phi(z)$ is piecewise linear. Indeed, for $z_i>0$ one has $d_i=1$ and
$\phi(z_i)=z_i$, whereas for $z_i<0$ one has
$d_i=0$ and $\phi(z_i)=0$. At points $z_i=0$,
both sides vanish for every $d_i\in[0,1]$.


\subsection{Matrix Structure}

In the following, we employ a semismooth Newton framework in which a single representative
$D(z)\in\partial \phi(z)$ is selected at each iterate and held fixed in the
linearization.

Differentiation of the transformed augmented Lagrangian yields the coupled system
\[
\nabla \DLag = -M(D(z))s,
\]
with
\[
M(D(z))=
\begin{pmatrix}
M_{yy} & M_{yu} & M_{y\lambda} & M_{y\mu}
\\
M_{uy} & M_{uu} & M_{u\lambda} & M_{u\mu}
\\
M_{\lambda y} & M_{\lambda u} & M_{\lambda\lambda} & 0
\\
M_{\mu y} & M_{\mu u} & 0 & M_{\mu\mu}
\end{pmatrix}.
\]

For the sake of simplicity, we omit the definition of the standard blocks and refer to the analogous work in \cite{walther2018equality}.

The inequality-dependent blocks are given by
\begin{align}
M_{\mu y}
&=
-D(z)k_y,
\\
M_{\mu u}
&=
-D(z)k_u,
\\
M_{\mu\mu}
&=
\frac1\alpha(I-D(z)).
\end{align}

\begin{theorem}
The gradient of the doubly augmented Lagrangian satisfies
\begin{equation}
\nabla \DLag = -M(D(z))s,
\label{eq:mainrepresentation}
\end{equation}
where the matrix $M$ preserves the structural block form of the equality-constrained framework and incorporates nonsmooth active-set effects through the generalized Jacobian $D(z)$.
\end{theorem}

\begin{proof}
Throughout the derivation we use $z=\alpha k(y,u)+\mu$.
After elimination of the slack variables, the transformed augmented
Lagrangian becomes
\begin{align}
\label{eq:prooflag}
\DLag(y,u,\lambda,\mu)
&=
f(y,u) + \lambda^\top(G(y,u)-y) + \frac{\alpha}{2}\|G(y,u)-y\|^2
\\
&\quad
+ \frac{1}{2\alpha} \| \phi(z) \|^2
- \frac{1}{2\alpha}\|\mu\|^2
+ \frac{\beta}{2} \| \nabla_y\Lag(y,u,\lambda,\mu) \|^2.
\nonumber
\end{align}

Differentiation with respect to the state variables yields
\begin{align*}
\nabla_y \DLag
&=
f_y
+
(G_y-I)^\top\lambda
\\
&\quad
+
\alpha(G_y-I)^\top(G-y)
\\
&\quad
+
k_y^\top D(z)\,z
\\
&\quad
+
\beta N_{yy}\nabla_y\Lag
+
\beta N_{yu}\nabla_u\Lag.
\end{align*}

Similarly, differentiation with respect to the design variables yields
\begin{align*}
\nabla_u \DLag
&=
f_u
+
G_u^\top\lambda
\\
&\quad
+
\alpha G_u^\top(G-y)
\\
&\quad
+
k_u^\top D(z)\,z
\\
&\quad
+
\beta N_{uy}\nabla_y\Lag
+
\beta N_{uu}\nabla_u\Lag.
\end{align*}

Using semismooth calculus, we apply the generalized chain rule and select
$D(z)\in \partial \phi(z)$, which is held fixed at the current iterate.

The directional derivative of the inequality term with frozen active set selection yields contributions of the form
\[
k_y^\top D(z)\,z, \qquad k_u^\top D(z)\,z.
\]
Using $\phi(z)=D(z)z$ and $D(z)^2=D(z)$,
we obtain
\[
\nabla_\mu
\left(
\frac1{2\alpha}\|\phi(z)\|^2
-\frac1{2\alpha}\|\mu\|^2
\right)
=
D(z)k
+\frac1\alpha(D(z)-I)\mu
\]
as a contribution to $\nabla_\mu \DLag. $

Collecting all contributions and introducing the one-shot step $s$ given by \eqref{eq:os_step}
yields the coupled block system
\[
\nabla \DLag = -M(D(z))s.
\]

All remaining blocks coincide structurally with the
equality-constrained framework of \cite{walther2018equality}.

\end{proof}

\section{Second-Order Subdifferential Analysis}

The transformed doubly augmented Lagrangian belongs to $C^1$, but is
not twice differentiable at points satisfying
\[
\alpha k+\mu=0.
\]

To analyze local optimality and stability properties of the transformed augmented Lagrangian, generalized second-order subdifferentials in the sense of Clarke \cite{clarke1983} and Rockafellar--Wets \cite{rockafellar2009} are employed.

We use the definitions of $z$ and $\phi(z)$ as introduced in Chapter 4. 


For every admissible generalized derivative selection
$D(z)\in\partial\phi(z)$, we associate a representative matrix
$H(D)$ obtained from the semismooth linearization of the transformed augmented Lagrangian.

Thus, each element of the generalized Hessian corresponds to a fixed active-set selection $D$. 

As a result, for every admissible generalized derivative selection
$D(z)\in\partial\phi(z)$ we associate a representative
generalized Hessian matrix
\[
H(D).
\]

We aim to show that for a fixed selection $D$, the corresponding Hessian $H(D)$ is positive definite. 

We closely follow the idea of \cite{walther2018equality}. The generalized Hessian is decomposed into
\[
\partial^2\DLag
=
\tilde H_1+\tilde H_2(D)+\tilde H_3(D).
\]
The matrix $\tilde H_1$ contains the dominant second-order terms inherited from the equality-constrained one-shot framework together with the multiplier block induced by the inequality constraints. The matrix $\tilde H_2(D)$ collects all active-set dependent contributions generated by the semismooth linearization. Finally, $\tilde H_3$ contains higher-order coupling terms originating from nonlinearities of $G$, $k$, and the adjoint residual. These terms coincide structurally with the perturbation terms appearing in the equality-constrained analysis.

In the following, we show that $\tilde H_2$ is positive semidefinite, then show that $\tilde H_1+\tilde H_2$ is positive definite and finally verify that $\tilde H_3$ does not affect coercivity. We only analyze the structure of the new nonsmooth contribution, which is the contribution of the terms \eqref{eq:maxterm}, i.e., 
\begin{align} 
\Phi(y,u,\mu) = \frac1{2\alpha} \| \phi(z) \|^2 - \frac1{2\alpha}\|\mu\|^2 
\end{align}. 
We split the corresponding Hessian contributions into two quadratic parts and a remainder, such that
$\Phi'' = \Phi''_{1}+\Phi''_{\mathrm{2}}(D)+\Phi''_{\mathrm{3}}(D).$

\subsection{Structure of the active-set dependent contributions}

The decomposition of $\Phi''$ induces the corresponding decomposition of the inequality-dependent part of the generalized Hessian. Accordingly, the active-set contribution of the full generalized Hessian is identified with
\[
\tilde H_2(D):=\Phi''_2(D).
\]

\begin{proposition}
\label{prophtwo}
    The active-set dependent contribution 
    \begin{align}
        \tilde H_2(D) = \Phi''_{2} = \begin{pmatrix} \alpha k_y^\top D k_y & \alpha k_y^\top D k_u & k_y^\top D \\[1ex] \alpha k_u^\top D k_y & \alpha k_u^\top D k_u & k_u^\top D \\[1ex] D k_y & D k_u & \frac1\alpha D \end{pmatrix}
    \end{align} 
    generated by the inequality term is positive semidefinite for every admissible generalized derivative selection $D(z)\in\partial\phi(z)$.
\end{proposition}

\begin{proof}
    Introducing \[ R= \begin{pmatrix} \sqrt{\alpha}\,D^{1/2}k_y & \sqrt{\alpha}\,D^{1/2}k_u & \frac1{\sqrt{\alpha}}D^{1/2} \end{pmatrix}, \] one obtains the factorization \[ \Phi''_{2} = R^\top R. \]

    Since $D(z)$ is diagonal with entries satisfying $ 0\le d_i\le 1, $ it follows that $D(z)\succeq0$ and therefore \[ \Phi''_{2} \succeq0. \]
\end{proof}

\subsection{Positive definiteness of the quadratic contributions}

We now consider the two quadratic contributions $\tilde{H}_1$ and $\tilde{H}_2(D)$.

The contribution
\[
\Phi''_1
=
\begin{pmatrix}
0&0&0\\
0&0&0\\
0&0&-\frac1\alpha I
\end{pmatrix}
\]
coincides structurally with the multiplier contribution arising in the
equality-constrained framework of \cite{walther2018equality}. Hence it
is incorporated into the dominant part $\tilde H_1$ of the generalized
Hessian.

Analogously to \cite{walther2018equality}, we define

\[
\tilde H_1=
\begin{pmatrix}
H_{11} & 0
\\
0 &
\beta k_y k_y^\top-\frac1\alpha I
\end{pmatrix},
\]

where $H_{11}$ denotes the dominant block associated with the state, design and adjoint variables.

\begin{proposition}
\label{prop:H1}
Assume
\begin{align}
    \alpha\beta \Delta G_y^\top \Delta G_y \succ I+\beta N_{yy},
\end{align}
and
\begin{align}
    \alpha\beta k_yk_y^\top \succ I
\end{align}
Then
\[
\tilde H_1\succ0.
\]
\end{proposition}

\begin{proof}

The positivity of the block $H_{11}$ follows from the analysis
of \cite{walther2018equality}.

For the multiplier block we obtain
\[
\beta k_yk_y^\top-\frac1\alpha I = \frac1\alpha \Bigl(
\alpha\beta k_yk_y^\top-I \Bigr).
\]

By assumption
\[
\alpha\beta k_yk_y^\top>I.
\]

Hence,
\[
\beta k_yk_y^\top-\frac1\alpha I
\succ0.
\]

Since $\tilde H_1$ is block diagonal, positivity of both diagonal blocks implies

\[
\tilde H_1\succ0.
\]

\end{proof}

The condition
\[
\alpha\beta k_yk_y^\top\succ I
\]
implicitly requires that the active inequality constraint gradients are
linearly independent. In particular, we assume that $k_y$ has full row
rank in a neighborhood of the solution.

We can now summarize these results in the following theorem.

\begin{theorem}
\label{thm:H1H2}

Assume that
\begin{align}
    \alpha\beta \Delta G_y^\top \Delta G_y \succ I+\beta N_{yy},
\end{align}
and
\begin{align}
    \alpha\beta k_yk_y^\top\succ I.
\end{align}

Then, for every admissible generalized derivative selection
$D(z)\in\partial\phi(z),$ the matrix $\tilde H_1+\tilde H_2(D)$ is positive definite.

\end{theorem}

\begin{proof}

By Proposition~\ref{prop:H1},
\[
\tilde H_1\succ0.
\]

Furthermore, Proposition~\ref{prophtwo} shows that
\[
\tilde H_2(D)\succeq0
\]
for every admissible generalized derivative selection
\[
D(z)\in\partial\phi(z).
\]

Since $\tilde H_1\succ0$, there exists a constant
$\gamma>0$ such that
\[
v^\top\tilde H_1v
\ge
\gamma\|v\|^2
\]
for all $v$.

Therefore,
\[
v^\top(\tilde H_1+\tilde H_2(D))v
=
v^\top\tilde H_1v
+
v^\top\tilde H_2(D)v
\ge
\gamma\|v\|^2.
\]

Hence
\[
\tilde H_1+\tilde H_2(D)\succ0.
\]

\end{proof}

\subsection{Higher-order remainder}

The remaining second-order contributions originate from the derivatives of the constraint Jacobians. Collecting all terms that are not contained in $\Phi''_{1}$ or
$\Phi''_{2}(D)$ yields

\[
\Phi''_{3}(D)
=
\begin{pmatrix}
(\nabla_y k_y^\top)Dz &
(\nabla_u k_y^\top)Dz &
0
\\[1ex]
(\nabla_y k_u^\top)Dz &
(\nabla_u k_u^\top)Dz &
0
\\[1ex]
0 & 0 & 0
\end{pmatrix}.
\]

The matrix $\Phi''_{3}(D)$ contains only second derivatives
of the inequality constraint mapping $k$. Since all entries are multiplied by $Dz$, these terms vanish at points satisfying $Dz=0$. Consequently, $\Phi''_{3}(D)$ acts as a higher-order perturbation of the dominant quadratic contributions.

We analyze the different cases that can occur. 

\paragraph{Case 1: Strictly Inactive Constraint}
Assume
$
0>\alpha k+\mu.
$

Then
\[
D(z)=0,
\]
hence $\Phi''_3=0$. 

\paragraph{Case 2: Strictly Active Constraint}

Assume
$
0<\alpha k+\mu.
$

Then
\[
D(z)=I.
\]

In this case, we can only conclude that, assuming bounded second derivatives of $k$, there exists
a constant $C>0$ such that
\begin{align*}
    ||\Phi''_3(I)||\leq C ||z||.
\end{align*}

\paragraph{Case 3: Active-Set Boundary}

Assume
$
0=\alpha k+\mu.
$

At the active-set boundary the generalized derivative becomes set-valued, i.e., 
\[
D(z)\in[0,I].
\]
However, due to $z=0$ we obtain directly $\Phi''_3(D)=0.$

Consequently, active-set transitions themselves do not affect the local coercivity analysis. The only nontrivial perturbations arise from strictly active constraints.

The previous case distinction shows that
$\tilde H_3(D)$ vanishes identically for inactive constraints and at active-set boundaries. Hence, only strictly active constraints may contribute to the perturbation term. In this regime,
$\tilde H_3(D)$ remains of order $\mathcal O(\|z\|)$ and therefore can be controlled locally.


\subsection{Result on Positive Definiteness}

\begin{theorem}
\label{thm:fullhessian}

Assume the hypotheses of Theorem~\ref{thm:H1H2}.
Further assume that the second derivatives of $k$ are bounded in a neighborhood of the solution.

Then there exists a neighborhood $\mathcal U$ of the solution such that  
\[
H(D)
=
\tilde H_1+\tilde H_2(D)+\tilde H_3(D)
\succ0
\]
for every admissible generalized derivative selection
$D\in\partial\phi(z)$ and every iterate in $\mathcal U$.
\end{theorem}

\begin{proof}

We follow the ideas of \cite{walther2018equality}. 
By Theorem~\ref{thm:H1H2} there exists a constant
$\gamma>0$ such that
\[
\lambda_{\min} \bigl(\tilde H_1+\tilde H_2(D)
\bigr) \ge \gamma
\]
for every admissible generalized derivative selection
$D\in\partial\phi(z)$.

Furthermore, the previous discussion shows that
\[
\|\tilde H_3(D)\| \le C\|z\|.
\]

Hence, for sufficiently small $\|z\|$,
\[
\|\tilde H_3(D)\| < \gamma .
\]

Since
\[
\lambda_{\min}(\tilde H_1+\tilde H_2(D))
\ge \gamma,
\]
uniformly in $D\in\partial\phi(z)$, the positive definiteness estimate
is independent of the particular active-set selection.

Weyl's eigenvalue perturbation theorem therefore yields
\[
\lambda_{\min}(H(D)) \ge
\lambda_{\min}
\bigl(
\tilde H_1+\tilde H_2(D)
\bigr)
-
\|\tilde H_3(D)\|
>
0.
\]

Consequently,
\[
H(D)\succ0.
\]
\end{proof}

\section{Conclusions and Outlook}

We extended the doubly augmented one-shot framework to optimization problems involving additional inequality constraints. After elimination of the slack variables, the resulting formulation preserves the characteristic coupled one-shot matrix structure while introducing nonsmooth active-set dependent contributions. Using generalized second-order subdifferentials in the sense of Clarke and Rockafellar--Wets, explicit positivity conditions guaranteeing positive definiteness of all admissible representative generalized
Hessians, providing second-order sufficient conditions for strict local
optimality of the transformed doubly augmented Lagrangian, were established.

The analysis reveals that the additional inequality constraints enter
the positivity conditions only through the active constraint Jacobian
$k_y$. In particular, the condition
\[
\alpha\beta k_yk_y^\top\succ I
\]
can be enforced through suitable choices of the augmentation parameters
provided that the active constraints satisfy a full-rank condition.

Future work will focus on semismooth local convergence theory, the construction of preconditioners, and on
large-scale PDE-constrained applications.

The idea here is to define a preconditioner $B$ for the design update and a precondition $\check B$ for the constraint multiplier update given by
\[
\check B = \beta k_yk_y^\top + \frac1\alpha(D(z)-I) +
\varepsilon I.
\]

In the equality constraint setup, the multiplier update is approximated by
\[
\check B\Delta\mu
\approx
\nabla_\mu \tilde L^a(\mu+\Delta\mu)
-
\nabla_\mu \tilde L^a(\mu),
\]
where we fix all other variables of the augmented Lagrangian and use finite differences for the derivative with respect to $\mu.$ We aim to expand this idea to the problem with inequality constraints. 


\section*{Acknowledgements}
This work was funded by the Deutsche Forschungsgemeinschaft (DFG, German Research Foundation)
– GRK 2982, 516090167, ``Mathematics of Interdisciplinary Multiobjective Optimization''.

\end{document}